\documentclass[11pt]{amsart}

\usepackage[all]{xy}
\usepackage{amsmath, amssymb, amsfonts, hyperref}
\newcommand{\pp}[1]{{\mathfrak #1}}

\newcommand{\A}{\operatorname{A}}

\newcommand{\divs}{\operatorname{div}}

\newcommand{\ext}{\operatorname{Ext}}
\newcommand{\tor}{\operatorname{Tor}}

\newcommand{\Hom}{\operatorname{Hom}}

\newcommand{\ch}{\operatorname{ch}}
\newcommand{\K}{\operatorname{K_0}} 
\newcommand{\td}{\operatorname{td}}

\newcommand{\depth}{\operatorname{depth}}
\newcommand{\Min}[1]{\operatorname{Min(#1)}}

\newcommand{\hk}[1]{\varphi_n(#1) } 
\newcommand{\hkr}[2]{\varphi_n^{#1}(#2)} 

\theoremstyle{plain}
\newtheorem{theorem}{Theorem}[section]

\newtheorem{lemma}[theorem]{Lemma}

\newtheorem{corollary}[theorem]{Corollary}

{\Alph{theorem}}

\theoremstyle{definition}

\newtheorem{example}[theorem]{Example}

\newtheorem{chunk}[theorem]{}

\newtheorem{remark}[theorem]{Remark}

\theoremstyle{remark}

\numberwithin{equation}{section}

\title[Hilbert-Kunz Functions over Rings Regular in Codimension one]
{Hilbert-Kunz Functions over Rings Regular in Codimension one}
\author{C-Y. Jean Chan and Kazuhiko Kurano}
\date{July~1, 2013 \\
  \indent 2010 {\em Mathematics Subject Classification.}  13A35, 13B22, 13D40, 14C15. \\ 
  \indent {\em Key words and phrases.} Hilbert-Kunz function, rational
  equivalence.  \\
  \indent The first author was partially supported by AWM-NSA Mentoring
  Travel Grant, and by FRCE Type B Grant \#48780 and Early Career
  Investigator's Grant \#C61368 of Central Michigan University. The
  second author was partially supported by KAKENHI (24540054).  }

\begin{document}

\maketitle

\begin{abstract}
 
The aim of this manuscript is to discuss the Hilbert-Kunz
functions over an excellent local ring regular in codimension one.
We study the shape of the Hilbert-Kunz functions of modules and discuss the properties
of the coefficient of the second highest term in the function. 

Our results extend Huneke, McDermott and Monsky's result about the shape of 
the Hilbert-Kunz functions in \cite{HMM} and a theorem of the second author in \cite{K5}
for rings with weaker conditions. In this paper, for a Cohen-Macaulay ring, 
we also explores an equivalence condition 
under which the second coefficient vanishes whenever the Hilbert-Kunz function of the ring
is considered with respect to an $\pp m$-primary ideal of finite projective dimension. 
We introduce an additive error of the Hilbert-Kunz functions of modules on a short exact 
sequence and give an estimate of such error.

\end{abstract}

\bigskip


\section{Introduction.}\label{open}

The aim of this paper is to study the stability of the Hilbert-Kunz
functions for finitely generated modules over local rings of prime
characteristic $p$. Besides its mysterious leading coefficient, known as 
the Hilbert-Kunz multiplicity, the
behavior of the Hilbert-Kunz function is rather unpredictable.
Some of the results presented in this paper are extensions of those 
in \cite{HMM} and \cite{K5} for rings with weaker
conditions. Following 
these improvements, we make further discussions on the properties of
the second terms in which the Hilbert-Kunz functions stabilize and 
estimate the {\em additive error} of the Hilbert-Kunz functions on the
short exact sequences (defined in Section~\ref{error}).

We begin with a brief introduction to the definition of the
Hilbert-Kunz functions and the motivation of our work presented in this paper. 

Throughout the paper, let $R$ be a Noetherian local ring of positive
characteristic $p$ and dimension $d$. We assume also that the residue
field of $R$ is perfect.  Let $I$ be an ideal primary to the maximal
ideal.  A {\em Frobenius $n$-th power} of $I$, denoted $I^{[p^n]}$, is
the ideal generated by all elements in the form of $x^{p^n}$ for any
$x$ in $I$. For simplicity on notation, we write $I_n$ for
$I^{[p^n]}$. The length of a finitely generated module, if exists, is
denoted by $\ell( \cdot )$.

Let $M$ be a finitely generated $R$-module.  In 1969
Kunz introduced a map from $\mathbb N$ to $\mathbb Z_{\geq 0}$:
for any positive integer $n$, define
\[ \hkr{R,I}{M} = \ell(M/I_n M). \] 
We note that the input variable $n$ is written 
as a subscript in the above expression for the convenience of discussion. 
This map was named the
{\em Hilbert-Kunz function of $M$ with respect to $I$} by
Monsky~\cite{Mon1}. Although the function depends on both $M$ and
$I$, when there is no ambiguity on the ideal $I$, we simply say the
Hilbert-Kunz function of $M$ and denote it by $\hk M$. 
Monsky considered also the limit
\begin{equation}\label{HKmult}
 \lim_{n \rightarrow \infty} \frac{ \hk{M}}{(p^{n})^t} 
\end{equation}
where $t=\dim M$ and proved the following results:

\begin{theorem}[Monsky~\cite{Mon1}]\label{Monsky} 
Let $(R,\pp m)$ be a Noetherian local ring of positive characteristic
$p$ and $\dim R=d$. Let $I$ be an $\pp m$-primary ideal and $M$ a
finitely generated module of dimension $t \leq d$. Then

$(a)$ The limit in (\ref{HKmult}) always exists as a positive real, 
denoted $e_{HK}(M)$. 
  
$(b)$ The Hilbert-Kunz function is always in the form of 
\begin{equation*}
\hk{M} = e_{HK}(M) q^t + O(q^{t-1}),
\end{equation*}
where $q = p^n$.

\end{theorem}

Monsky named the  limit in (\ref{HKmult}) the {\em Hilbert-Kunz multiplicity
  of $M$} ({\em with respect to $I$}).  We recall that a function
$f(n)$ is $O(q^s)$ for $q=p^n$ and some integer $s$ if there exists a constant $C$ such that
$|f(n)| \leq C \cdot (p^n)^s$ for all $n \gg 0$.

In this paper, we express the Hilbert-Kunz function of $M$ as
\begin{equation}\label{HKMon}
\hk M = \alpha_I(M) q^d + O(q^{d-1}),
\end{equation}
where $\alpha_I(M)$ equals $e_{HK}(M)$ if $\dim M =d$ and zero otherwise. 
Huneke, McDermott and Monsky studied further how 
$ \hk M$ depends on $n$ as $n$ grows. Their main theorem states: 

\begin{theorem}[Huneke-McDermott-Monsky~\cite{HMM}] \label{thmHMM}
Let $(R, \pp m)$ be an excellent local normal domain of characteristic $p$
with a perfect residue field and $\dim R =d$. Then 
$\hk M = \alpha_I(M) q^d + \beta_I(M) q^{d-1} + \mathcal O( q^{d-2} ) $
for some constants 	$\alpha_I(M)$ and $\beta_I(M)$ in $\mathbb R$. 
\end{theorem} 

If the module $M$ and the primary ideal $I$ is clear to the content
under consideration, we drop them from the expression and use $\alpha$
and $\beta$ in place of $\alpha_I(M)$ and $\beta_I(M)$ respectively. 

In this paper we say that a local ring satisfies (R1$'$) or is {\em
  regular in codimension one} if the localization of the ring is a
field (resp. a DVR) at a prime ideal of dimension $d$
(resp. $d-1$). Here the dimension of a prime ideal $\pp p$ means the
Krull dimension of $R/\pp p$. Since this condition is very similar 
to the usual (R1),  we denote it by (R1$'$).

We observe that a key lemma and certain crucial properties needed to
prove Theorem~\ref{thmHMM} either hold or have their natural
substitution without the normal condition on the ring, although 
 these extensions are not necessarily obvious (see Section~\ref{domain}
for details). Therefore it is natural to ask: is the normal condition
essential for $\hk M$ to stabilize at the second term $\beta_I(M)
q^{d-1}$? 

The goals of this paper are to prove that the shape of the
Hilbert-Kunz functions of modules in Theorem~\ref{thmHMM} holds for
excellent local rings satisfying (R1'), and to analyze
vanishing properties of the second term. The latter extends a theorem
of the second author in \cite{K5} about the vanishing of $\beta_I(R)$
related to the canonical module of $R$. We also discuss the vanishing
properties of $\beta_I(R)$ in terms of the Todd class of the ring in
the numerical Chow group. Furthermore, we prove that associated to
each module, there exists another quantity $\tau$ which is 
additive on short exact sequences.  The Hilbert-Kunz
multiplicity is additive on short exact sequences but the Hilbert-Kunz
function is not. Using the additivity of $\tau$, we define an additive
error of the Hilbert-Kunz functions, and provide this error an
estimation in terms of the torsion submodules and the $\tau$
value of appropriate modules.

It should be noted that generalization of the work in \cite{HMM} is
also studied independently by Hochster and Yao~\cite{HY} via a
different approach.

An example in Monsky~\cite{Mon1} shows that the (R1$'$) condition can not
be further relaxed for such stability of the Hilbert-Kunz function to
hold. Precisely, take $R = \mathbb Z/p[[x,y]]/(x^5 -y^5)$ then $R$ is
one-dimensional, has an isolated singularity and its Hilbert-Kunz
function is $\hk R = 5 p^n + \delta_n$ where $\delta_n = -4$ if $n$ is 
even and $-6$ if $n$ is odd.

As mentioned also in \cite{HMM} the normal ring $R = \mathbb Z/5
\mathbb Z [x_1, x_2, x_3, x_4]/(x_1^4+ x_2^4 + x_3^4 + x_4^4)$ of
dimension $d=3$ has the Hilbert-Kunz function $\hk R=\frac{168}{61}
(5^n)^3 - \frac{107}{61}(3^n)$ which was computed by Han and Monsky
\cite{HaM}. The tail $-\frac{107}{61}(3^n)$ is $O((5^n)^{d-2})$ but
not $\mathcal O((5^n)^{d-3})$. So the existence of the third
coefficient in the Hilbert-Kunz function in general is not possible.

We briefly describe the outline of the paper and the machinery used in
each section as follows. 

Section~\ref{chow} reviews the
definition of the rational equivalence and properties of cycle classes
that will be used in the later discussions. 

Section~\ref{general} contains the main results of this paper about
the existence of the second coefficient of the Hilbert-Kunz function
and its property. 
First we prove that the shape of the Hilbert-Kunz function as
described in Theorem~\ref{thmHMM} holds for local excellent rings that
satisfy (R1$'$) condition in Theorem~\ref{thm.HK}.  The method of
proving Theorem~\ref{thm.HK} is to reduce general cases to normal
domains and then apply Theorem~\ref{thmHMM}.  In the proof of
Theorem~\ref{thm.HK}, no rational equivalence is involved. We prove
also that the vanishing property of the second coefficient of the
Hilbert-Kunz function with respect to the maximal ideal is
characterized by the canonical module in Theorem~\ref{thm.can} . These
generalize the result of Theorem~\ref{thmHMM} and that in \cite{K5}
respectively. Secondly we investigate the properties of the second
coefficient with respect to arbitrary $\pp m$-primary ideals of finite
projective dimension. This is done in Theorem~\ref{CMcase} using the
techniques developed in \cite{K16, K5, Sr}.

Section~\ref{error}
presents some possible applications. Theorem~\ref{tau} first proves
that each torsion free module is associated with a real number $\tau$
and then formally describes $\tau$ as a group homomorphism compatible
with rational equivalence (extending
\cite[Corollary~1.10]{HMM}). Finally we deduce that $\tau$ is additive
on a short exact sequence and that the {\em additive error of the
  Hilbert-Kunz function} always arises from torsion submodules.

Finally in Section~\ref{domain} we revisit the proof in \cite{HMM} and
show that all delicate analysis in \cite{HMM} works for an arbitrary local
domain that is $F$-finite and satisfies (R1$'$). This section is listed
as an appendix since the result is a special case of Theroem~\ref{thm.HK}.
Nevertheless the section provides a straightforward proof
without the need of taking normalization. 
We believe that such a generalized proof of \cite{HMM} consists of
interesting arguments and it is worth sharing it with curious readers;
especially those who are interested in rational equivalence. This
argument, however, has its own limitation. It works for integral domains. 
The authors do not know how to extend the analysis beyond the case of 
integral domains.

In the proofs of the main theorems, it is essential that the integral
closure of the ring $R$ in its ring of fractions is finite over $R$.
For general results in Sections~\ref{general} and \ref{error} where $R$ is not
necessarily a domain, we assume $R$ to be excellent.  In
Section~\ref{domain} where the ring is an integral domain, we assume
$R$ is $F$-finite which implies $R$ is also excellent.

{\sl Acknowledgement.} The authors thank Roger Wiegand for pointing out an
error in the proof of Corollary~\ref{additive} in its early version, whose comment led 
to the corollary's current form.


\bigskip

\section{Preliminary on the Chow Group}\label{chow}

In this section, we recall the definition of Chow groups following
Roberts~\cite{R1}, and state some properties that will be handy in the
later sections.  Let $R$ be a Noetherian ring of dimension $d$.  We
define $Z_i(R)$ to be the free Abelian group generated by all prime
ideals of dimension $i$ in $R$. The {\em group of cycles} is the direct
sum $Z_*(R) = \oplus_{i =0}^d Z_i(R)$. For any prime ideal $\pp p$, we
write $[R/\pp p]$ for the element corresponding to $\pp p$ in $Z_*(R)$.
Let $\pp q$ be a prime ideal of dimension $i+1$
and $x$ an element in $R$ not contained in $\pp q$. 
The {\em rational equivalence} is an equivalence relation on $Z_*(R)$ by 
setting $\divs(\pp q, x) =0$ where
\[ \divs(\pp q, x) = \sum \ell ( (R/\pp q) _{\pp p} / x (R/\pp q) _{\pp p} ) 
[R/\pp p] \] with the summation over all prime ideals in $R$ of
dimension $i$, so $\divs(\pp q, x)$ is an element in $Z_i(R)$.  
Note that this is a finite sum since there are only
finitely many minimal prime ideals for $R/xR$. 
 Let ${\rm Rat}_i(R) $ be the subgroup of $Z_i(R)$
generated by $\divs(\pp q, x)$ for all $\pp q$ of
dimension $i+1$ and all $x \in R - \pp q$.  The {\em Chow group}
$A_*(R)$ of $R$ is the quotient of $Z_*(R)$ by ${\rm Rat}_*(R) =
\oplus _{i=0}^d {\rm Rat}_i(R)$.  The Chow group is also decomposed into the
direct sum of $\A_i(R)=Z_i(R)/ {\rm Rat}_i(R)$ for all $i= 0, \dots, d$. By
abuse of notation, we also use $[R/\pp p]$ to denote the image of 
$[R/\pp p]$ in $A_*(R)$.

For any finitely generated module $M$, there exists a prime filtration 
\[ 0 = M_0 \subset M_1 \subset \cdots \subset M_{n-1} \subset M_n=M \]
such that each quotient of consecutive submodules is a cyclic module
whose annihilator is exactly a prime ideal; {\em i.e.}, $M_{i+1}/M_i
\cong R/\pp p_i$ for some prime $\pp p_i$ and for all $i =0, \dots,
n$. We note that such a prime filtration of a module is not
unique. The prime ideals occurring in a filtration and the number of
times each prime ideal occurs vary except for the minimal prime
ideals.  Indeed if $\pp p$ is a minimal prime ideal for $M$, then the
number of times that it occurs in a filtration is exactly $\ell (M
_{\pp p} ) $.  It is proved in \cite{C1} that the sums of prime ideals
of dimension $d$ and $d-1$ from different filtrations are rationally
equivalent. Therefore they define a unique class in the Chow
group. Their equivalence class in $A_*(R)$ is called the {\em cycle
  class} of $M$ denoted $[M]$. By definition $[M] = [M]_d + [M]_{d-1}$
with $[M]_i \in A_i(R)$ for $i = d, d-1$.  Theorem~\ref{thmChan} lists
a property that will be utilized in this paper:

\begin{theorem}[Chan~\cite{C1}]\label{thmChan} 
Let $R$ be a Noetherian ring of dimension $d$ and let $M$ be a finitely generated 
module over $R$. 
Then the cycle class $[M] = [M]_d + [M]_{d-1}$ in $Z_d(R) \oplus
Z_{d-1}(R)$ defined by taking the sum of prime ideals of dimension $d$
and $d-1$ in a prime filtration has the following properties:

$(a)$ $[M]$ is independent of the choice of filtrations and hence
defines a unique class in $\A_*(R)$. 

$(b)$ If $0 \rightarrow M_1 \rightarrow M_2 \rightarrow M_3
\rightarrow 0$ is a short exact sequence of finitely generated
$R$-modules, then $[M_3]_i - [M_2]_i + [M_1]_i =0$ in $\A_i(R)$, for $i=
d$ and $d-1$.

\end{theorem} 

We make a few remarks on the cycle classes just defined: 
If $R$ is a domain and $M$ is a module
of rank $r$, then $[M]_d = r [R]$ in $\A_d(R)$. If $R$ is normal, then
$\A_{d-1}(R)$ is isomorphic to the divisor class group with $[M]_{d-1}$
mapped to $ \operatorname{cl}(M)$.


\bigskip

\section{Main Theorems}\label{general}

In this section, unless otherwise indicated, the ring $R$ is an
excellent local ring {\em regular in codimension one} by which we mean
that $R_{\pp p}$ is a field for every prime ideal $\pp p$ with $\dim
R/\pp p =d$ and $R_{\pp p}$ is a DVR for those $\pp
p$ with $\dim R/\pp p = d-1$. We sometimes denote by $\dim \pp p$ the
Krull dimension of the ring $R/\pp p$. As mentioned in Section~\ref{open}, in
this paper we also use (R1$'$) to denote this condition.

Theorem~\ref{thm.HK} proves, for any finitely generated module $M$,
the existence of the second coefficient of the Hilbert-Kunz function
$\hk M$ with respect to an arbitrary maximal primary ideal $I$.  This
generalizes the main result of Huneke, McDermott and Monsky~\cite{HMM}
where $R$ is assumed to be an excellent local normal domain.

Using the singular Riemann-Roch theorem, the second author in
\cite{K5} proves that in a Noetherian normal local domain $R$, if the
canonical module of $R$ is a torsion element in the divisor class
group, then the second coefficient of the Hilbert-Kunz function $\hk
R$ also vanishes. Theorem~\ref{thm.can} shows that this result holds
also in the more general setting of Theorem~\ref{thm.HK} in which the
second coefficient exists.

In preparation for the proofs of the main theorems, we begin with
two useful lemmas. 

\begin{lemma}\label{codim2} 
  Let $R$ be a Noetherian local ring and let $M_i$ be finitely
  generated modules over $R$ for $i=1, 2, 3, 4$.  Assume that $ 0
  \rightarrow M_1 \rightarrow M_2 \rightarrow M_3 \rightarrow M_4
  \rightarrow 0$ is exact and that $M_1$ and $M_4$ have dimension at
  most $d-2$. Then $\hk {M_2} = \hk {M_3} + O(q^{d-2})$.
\end{lemma} 

\begin{proof} 
  Let $N$ be the image of the map $M_2 \rightarrow M_3$. Then we
  obtain two exact sequences: $0 \rightarrow N \rightarrow M_3
  \rightarrow M_4 \rightarrow 0$ and $ 0 \rightarrow M_1 \rightarrow
  M_2 \rightarrow N \rightarrow 0$. It is enough to show $\hk{N} =
  \hk{M_3} +\mathcal O(q^{d-2})$ and $\hk{M_2} = \hk N +\mathcal
  O(q^{d-2})$. The equalities on the Hilbert-Kunz function are the
  results of the short exact sequences in their respective order. We
  give a proof for the first one. The second equality is based on a
  similar argument.

We tensor the short exact sequence
$ 0 \rightarrow N \rightarrow M_3 \rightarrow M_4 \rightarrow 0$ by 
$R/I_n$ and obtain
\[ \cdots \rightarrow \tor_1^R(M_4, R/I_n) \rightarrow N/I_nN
\rightarrow M_3/I_nM_3 \rightarrow M_4/I_n M_4 \rightarrow 0\]
which yields
\[ \hk{M_4} - \hk{M_3} + \hk{N} = \ell(K) \geq 0 \] where $K$ denotes
the image of the $\tor_1$-module in $N/I_nN$.  Furthermore $\ell(K)$
is bounded by $\ell(\tor_1^R(M_4, R/I_n) )$. Thus we obtain that
$\ell(\tor_1^R(M_4, R/I_n) ) = \mathcal O(q^t) $ with $t=\dim M_4$
({\em c.f.} \cite[Lemma1.1]{HMM} or Section~\ref{domain} of the current
  paper) and $\ell(K) = \mathcal O(q^{d-2})$. Also by the assumption
  and Theorem~\ref{Monsky}(b), $ \hk{M_4} =\mathcal O(q^{d-2})$. Hence
  $\hk{N} = \hk{M_3} +\mathcal O(q^{d-2})$ as desired.
\end{proof}

If  $N$ can be viewed as a finitely generated module over $R$ and
$S$ simultaneously, $\hkr{R, I}{N}$ (resp.  $\hkr{S, IS}{N} $) denotes the
Hilbert-Kunz function of N as a module over $R$ (resp. $S$) and we
skip the superscript when there is no ambiguity.

We consider the primary decomposition of the zero ideal of $R$: 
\begin{equation}\label{primary}
 (0) = \pp q_1 \cap \cdots \cap \pp q_u \cap \pp q_{u+1} \cap \cdots \cap
\pp q_{\ell} .
\end{equation} 
With the (R1) condition on the ring, one observes the following 
properties of this decomposition: 
\begin{itemize} 
\item Each primary ideal of dimension $d$ in the decomposition is a prime ideal.
\item There is not a primary ideal of dimension $d-1$ in the decomposition. 
\item Any prime ideal of $R$ of dimension $d-1$ contains a unique prime
  ideal of dimension $d$. 
\end{itemize} 

We assume that $\dim{\pp q_i} = d$ for $i=1, \dots, u$ and $\dim{\pp
  q_i} \leq d-2$ otherwise. We further observe that the Krull dimension of
the module $\pp q_1 \cap \cdots \cap \pp q_u $ is at most $d-2$ because $ \pp q_1
\cap \cdots \cap \pp q_u$ becomes zero when localizing at prime ideals
of dimension larger than $d-2$.  In fact let $\pp p$ be a prime ideal
of dimension at least $d-1$. Since $\pp p$ contains a unique
$d$-dimensional prime ideal in $R$, say $\pp q_1$, and since $R_{\pp
  p}$ is a regular local ring, then the localized ideal $\pp q_1
R_{\pp p}$ must be the zero ideal.

\begin{theorem}\label{thm.HK} 
  Let $(R, \pp m)$ be an excellent local ring of dimension $d$ and
  positive characteristic $p$ whose residue field is perfect. Assume
  that $R$ satisfies (R1$'$) condition.  Let $I$ be an $\pp m$-primary
  ideal.  Then there exist constant $\alpha(M)$ and $\beta(M)$ in
  $\mathbb R$ such that the Hilbert-Kunz function of $M$ with respect
  to $I$ is
\[ \hk M = \alpha(M) q^d + \beta(M) q^{d-1} + \mathcal O(q^{d-2}) . \]
\end{theorem} 

\begin{proof} 
 Since all conditions pass through completion and the Hilbert-Kunz
  function remains the same, we replace $R$ by its completion and assume
  that it is complete. 

Let $\pp q_1, \dots, \pp q_u$ be the minimal prime ideals of $R$ such that
$\dim R/\pp q_i = d$. The following sequence is
exact
\begin{equation}\label{reducedexact}
  0  \longrightarrow K \longrightarrow R  \stackrel{\eta}
  {\longrightarrow} \oplus_{i=1}^u \overline{R/\pp q_i} 
  \longrightarrow C \longrightarrow 0 ,
\end{equation} 
where $\eta$ is the composition of the usual projection of a ring to
its quotient followed by the inclusion to the integral closure of the
quotient, and $K$ and $C$ are the kernel and cokernel of $\eta$. Here
we remark that $\overline{R/\pp q_i}$ is a local normal domain since
$R/\pp q_i$ is a complete local domain. For
any prime ideal $\pp p$ of dimension at least $d-1$, $\pp p$ contains
exactly one of $\pp q_1, \dots, \pp q_u$, say $\pp q_1$. Moreover $\pp 
q_1 R_{\pp p}$ is the zero ideal since $R_{\pp p}$ is regular. This implies
$R_{\pp p} \cong (\oplus_{i=1}^u \overline{R/\pp q_i} )_{\pp p}$ and so
$K$ and $C$ have dimension at most $d-2$. 

For an arbitrary module $M$, (\ref{reducedexact}) induces the exact
sequence 
\[ 0 \longrightarrow K' \longrightarrow R \otimes M \stackrel{\eta
  \otimes 1}{\longrightarrow} \oplus_{i=1}^u (\overline{R/\pp q_i}
\otimes M ) \longrightarrow C' \longrightarrow 0 \] where $C' =
C\otimes M$ and $K'$ is the kernel of $\eta \otimes 1$, and they also
have dimension at most $d-2$. By Lemma~\ref{codim2}, we have
\[ \begin{array}{ccl}
\hkr{R,I}{M} & = & \sum_{i=1}^{u} \hkr{R,I}{ \overline{(R/\pp q_i)}
  \otimes M} + \mathcal O(q^{d-2}) \\ & = & 
\sum_{i=1}^{u} g_i \hkr{\overline{(R/\pp q_i)}, I \overline{(R/\pp
    q_i)} }   { \overline{(R/\pp q_i)} \otimes M} + \mathcal O(q^{d-2})
\end{array}
\]
where $g_i$ is defined to be the degree of the field extension 
$[ \kappa(\overline{(R/\pp q_i)} ) : \kappa(R)] $ in which $\kappa(\cdot)$
 denotes residue field of the local ring. 
Now each $\overline{ R/\pp
  q_i}$ is a normal local ring so the proof is completed by applying
Theorem~\ref{thmHMM}. 

\end{proof}

For the remaining of the section, we assume that $R$ is a homomorphic
image of a regular local ring $A$. We define the canonical module of
$R$ to be $\omega_R = \ext^c(R, A)$ where $c =\dim A - \dim R$ is 
the codimension. For $\overline R$, we put
$\omega_{\overline R} = \ext^c(\overline R, A)$ with $c= \dim A - \dim R$. Here
the definition of the canonical module follows from \cite[Satz~5.12]{HeK} (see
also \cite{Ao} and \cite[Remark~3.5.10]{BrH}).

\begin{theorem}\label{thm.can} 
Let $(R, \pp m)$ be as in Theorem~\ref{thm.HK}. 
Assume also that $R$ is the homomorphic image of a regular 
local ring $A$. Let $\omega_R$ be the canonical module of $R$. 
If $[\omega_R]_{d-1} =0$ in $\A_{d-1}(R)_{\mathbb Q}$, then 
$\beta(R) $ in the Hilbert-Kunz function $\hk R$ vanishes.
\end{theorem} 

\begin{proof}
As proving Theorem~\ref{thm.HK}, we also replace $R$ by its completion
and assume that $R$ is complete. Here note that we can show 
$[ \omega_R \otimes \hat{R}]_{d-1} =0$ using the (R1')
condition of $R$. 
We write the Hilbert-Kunz function of $R$ as $\hkr{R,I}{R} = \alpha
q^d + \beta q^{d-2} + \mathcal O(q^{d-2})$. 
Recall the proof of Theorem~\ref{thm.HK} and  for each $i$ let
$\alpha_i$ and $\beta_i$ be real numbers in the Hilbert-Kunz function
of $\overline{(R/\pp q_i)}$ such that
\[ \hkr{\overline{(R/\pp q_i)}, I \overline{(R/\pp q_i)} }
{\overline{(R/\pp q_i)}} = \alpha_i q^d + \beta_i q^{d-1} + \mathcal
O(q^{d-2}) . \] Then $\beta = \sum_i g_i \beta_i$ where $g_i = [
\kappa(\overline{(R/\pp q_i)} ) : \kappa(R)] $ also as defined in the
previous proof. Obviously in order to prove $\beta=0$, it suffices to
prove $\beta_i =0$ for each $i$.

{\em Claim:}  {\em That $[\omega_R]_{d-1} =0$ in $\A_{d-1}(R)_{\mathbb
  Q}$ implies $[\omega_{\overline{(R/\pp q_i)} } ]_{d-1}=0$ in $\A_{d-1}(
\overline{(R/\pp q_i)} )_{\mathbb Q}$.} 

Assume the claim. Since $\overline{(R/\pp
  q_i)}$ is a normal local ring and $[\omega_{\overline{(R/\pp
    q_i)}}]_{d-1}$ vanishes in $\A_{d-1}( \overline{ (R/\pp q_i) } )
_{\mathbb Q}$, then $\beta_i =0$ by
\cite[Corollary~1.4]{K5} and the theorem is proved. 

Now we prove the above claim.  Recall the minimal prime ideals $\pp
q_1, \dots, \pp q_u$ of $R$ and the map $\eta$ in
(\ref{reducedexact}). Since $R$ and $\oplus_{i=1}^{u} \overline{
  (R/\pp q_i) }$ are isomorphic when localizing at prime ideals of
dimension $\geq d-1$, $\eta$ also induces an isomorphism on
the $(d-1)$-component of the Chow groups
\begin{equation}\label{d-1chow}
\oplus_{i=1}^u    \A_{d-1}(\overline{ (R/\pp q_i) } ) \underset{\xi}{\simeq} \A_{d-1}(R) .
\end{equation}

We observe that 
$\omega_{ \overline{ (R/\pp q_i) } } \cong \Hom_R(\overline{ (R/\pp
  q_i) }, \omega_R)$. 
Indeed let $\mathbb I^{\bullet}$ be an injective resolution of $A$, then
there exists a quasi-isomorphism between the two complexes
$ \Hom_A( \overline{ (R/\pp q_i)  }, \mathbb
 I^{\bullet})$ and $ \Hom_R( \overline{ (R/\pp q_i) } , \Hom_A(R,
 \mathbb I^{\bullet}))$. Hence by the definition of the canonical
 modules we have the isomorphism just mentioned. 
Then we apply $\Hom( -, \omega_R)$ to the exact sequence
(\ref{reducedexact}) and obtain the following exact sequence
for some $C''$ and $K''$ of dimension at most $d-2$: 
\begin{equation}\label{canonicalexact}
 0 \longrightarrow C'' \longrightarrow \Pi_{i=1}^u \omega_{
  \overline{ (R/\pp q_i) }} \stackrel{\theta} \longrightarrow \omega_R \longrightarrow
K'' \longrightarrow 0 . 
\end{equation}
Since both $\xi$ in (\ref{d-1chow}) and $\theta$ in
(\ref{canonicalexact}) are induced by $\eta$, we may conclude that
\[ \xi( \oplus [ \omega_{ \overline{ (R/\pp q_i) } } ]_{d-1} ) = [\omega_R]_{d-1}. \]
Hence $ [\omega_R]_{d-1} =0$ if and only if $[ \omega_{ \overline{ (R/\pp
    q_i) } } ]_{d-1}  =0$ for each $i = 1, \dots, u$ and the proof of
the claim is completed.

\end{proof}

Let $C(R)$ be the category of bounded complexes of free $R$-modules
with support in $\{ \pp m \}$. We recall the definition of numerical
equivalence in Kurano~\cite{K16}. Define a
subgroup of $\A_*(R)_{\mathbb Q}$
\[ N\A_*(R)_{\mathbb Q} = \{ \gamma \in \A_*(R)_{\mathbb Q} |
  \ch(\alpha) \cap \gamma =0 \mbox{ for any } \alpha \in C(R) \} 
\] 
in which $\ch(\alpha) \cap \gamma$ is the intersection of the
localized Chern character of $\alpha$ with $\gamma$ in the sense of
\cite{F} or \cite{R1}.  An element in $\A_*(R)_{\mathbb Q}$ is said to
be {\em numerically equivalent} to zero if it is an element in
$ N\A_*(R)_{\mathbb Q}$. The group modulo the numerical
equivalence, denoted $\overline{ \A_*(R)_{\mathbb Q} }$, is called
{\em the numerical Chow group}. It is proved in \cite{K5} that
$N\A_*(R)_{\mathbb Q}$ maintains the grading that comes from the
dimension of cycles; that is, 
\[ N\A_*(R)_{\mathbb Q} = \oplus_i N\A_i(R)_{\mathbb Q}.\]

The singular Riemann-Roch theorem was first applied to study the
Hilbert-Kunz function in Kurano~\cite{K16, K5}. In the following
Lemma~\ref{RRhk}, we restate a result proved by a computational
technique presented in \cite[Example~3.1(3)]{K5} ({\em c.f.} also
\cite[Proposition~1]{Sei}) .

\begin{lemma}[\cite{K5,Sei}]\label{RRhk}  
Let $R$ be the homomorphic image of a regular local ring. We assume
that $R$ is an $F$-finite Cohen-Macaulay ring such that the
residue class field is perfect. Let $I$ be 
an $\pp m$-primary ideal of finite projective dimension. Then the
Hilbert-Kunz function of $R$ with respect to $I$ is a polynomial in $p^n$. 

Precisely, if we let $\mathbb G_{\bullet}$ be a finite free resolution
of $R/I$ and let $c_i $ be in $\A_i(R)_{\mathbb Q}$ such that the Todd
class $\td([R])$ is $ c_d + c_{d-1} + \cdots + c_1 + c_0$ in
$\A_*(R)_{\mathbb Q}$.  Then
\[ \hkr{R,I}{R}= \sum_{i=0}^d (\ch_i(\mathbb G_{\bullet}) \cap c_i)
q^i. \]

\end{lemma}

\begin{proof} 
By definition the Euler characteristic is the alternating sum of the length of
homology modules; that is 
\[ \chi_{\mathbb G_{\bullet} } ( R^{\frac{1}{p^n} }) = \sum_i (-1)^i
\ell ( H_i (\mathbb G_{\bullet} \otimes R^{\frac{1}{p^n} } ) ) . \]
Since the Frobenius functor is exact and when applied to a complex, we
simply raise the boundary maps by the corresponding Frobenius power,
therefore
\[ \begin{array}{ccl}
\chi_{\mathbb G_{\bullet} } ( R^{\frac{1}{p^n} }) 
& = & \ell ( H_0 (\mathbb G_{\bullet} \otimes R^{\frac{1}{p^n} } ) )
\\
& = & \ell ( R/ I^{ [p^n] } ) = \hkr{R,I}{R}. 
\end{array}
\]
On the other hand by the
singular Riemann-Roch theorem, we have 
\[\chi_{\mathbb G_{\bullet} } ( R^{\frac{1}{p^n} })  =  
\ch (\mathbb G_{\bullet} ) ( \td ( [ R^{\frac{1}{p^n}} ]) )
\]
It is known that $ \td ( [ R^{\frac{1}{p^n}} ]) $ decomposes in the
Chow group and that we have for the $i$-th component the formula
$\td_i ([R^{\frac{1}{p^n}}]) = p^{in} \td_i([R]) = p^{in} c_i$
(\cite[Lemma~2.2~({\em iii})]{K16}).
Then

\[ \begin{array}{ccl}
\chi_{\mathbb G_{\bullet} } ( R^{\frac{1}{p^n} }) & = & 
\ch (\mathbb G_{\bullet} ) ( \td ( [ R^{\frac{1}{p^n}} ]) ) \\
& = & \ch (\mathbb G_{\bullet}) ( p^{dn} c_d + \cdots + p^n
c_1 + c_0 ) \\
& = & \sum_{i=0}^d (p^n)^i \ch_i(\mathbb G_{\bullet}) \cap c_i .
\end{array}
\]

Notice that $\ch_i(\mathbb G_{\bullet}) \cap c_i $ is in $\A_*(R/\pp
m)_{\mathbb Q} \simeq \mathbb Q$. It is clear now that the Hilbert-Kunz
function of $R$ with respect to $I$  is a polynomial in $q = p^n$ with
coefficients described by the intersection of Tood classes and localized Chern characters
\[ \hkr{R,I}{R} = \chi_{\mathbb G_{\bullet} } ( R^{\frac{1}{p^n} }) =
\sum_{i=0}^d (\ch_i(\mathbb G_{\bullet}) \cap c_i) (p^n)^i =
\sum_{i=0}^d (\ch_i(\mathbb G_{\bullet}) \cap c_i) q^i. \]

\end{proof}

\begin{theorem}\label{CMcase}
  Let $R$ be the homomorphic image of a regular local ring. We assume
  that $R$ is a Cohen-Macaulay ring and is $F$-finite such that the
  residue class field is perfect. Then the
  following statements are equivalent: \\ $(a)$ $\beta_I(R) =0$ for any
  $\pp m$-primary ideal $I$ of finite projective dimension. \\ $(b)$
  $\td_{d-1}([R]) =0$ in $\overline{\A_{d-1}(R)_{\mathbb Q} }$ where
  $\td_{d-1}([R])$ is the $(d-1)$-component of the Todd class of $[R]$.

In addition if $R_{\pp p}$ is Gorenstein for all minimal prime ideals
$\pp p$ of $R$, then $(a)$ and $(b)$ are equivalent
to $(c)$: \\
$(c)$ $[\omega_R]_{d-1}= [R]_{d-1}$ in
$\overline{\A_{d-1}(R)_{\mathbb Q}}$.
\end{theorem}

\begin{proof}
The equivalence of ($a$) and ($b$) will be proved by using the method in the
proof of Theorem~6.4(2) in \cite{K16}. We give a complete proof here. 
Lemma~\ref{RRhk} gives a presentation for the second coefficient of $\hkr{R,I}{R}$:
\[ \beta_I(R) = \ch_{d-1}(\mathbb G_{\bullet}) \cap c_{d-1} .\] Assume
(b); {\em i.e.}, $c_{d-1}=\td_{d-1}([R])=0$ in $\overline{
  \A_*(R)_{\mathbb Q}}$.  Then by the definition of the numerical
equivalence, $\ch_{d-1}(\mathbb G_{\bullet}) \cap c_{d-1}= 0$ since
$\mathbb G_{\bullet} \in C(R)$. Hence $\beta_I(R) =0$ and this proves
(b) $\Rightarrow$ (a).

Conversely, assume (a). Let $\K(C)$ be the Grothendieck group of the
category $C$ of $R$-modules that have finite projective dimension and
finite length. For any $M$ in $C$, there
exist finitely many ideals $I_1, \dots, I_{\ell}$ generated by maximal
regular sequences and an $\pp m$-primary ideal $I$ of finite projective
dimension such that
\begin{equation}\label{Srinivas}
 [M] + \sum_{i=1}^{\ell} [R/I_i] = [R/I] . 
\end{equation}
in $\K(C)$.

We see (\ref{Srinivas}) by applying Kumar's proof for Hochster's
theorem on $\K(C)$ (\cite[Lemma~9.10]{Sr}). Let $\pp a$ be the
annihilator of $M$. Then there exists a regular sequence $f_1, \dots,
f_d$ of maximal length in $\pp a$. 
If $M$ is not cyclic, we assume $M$ is minimally generated by $n \geq
2$ elements, and let $x_1$ and $x_2$ be part of a minimal generating
set of $M$. We may construct a homomorphism $\varphi: (f_1, \dots,
f_d) \rightarrow M$ such that $\varphi(f_i) =x_i$ for $i=1,2$. This
can be done since the $f_i$'s are in the annihilator of $M$. Let $N$ be the
push-out defined by $\varphi$; that is, $N = (M \oplus R)/ B$ where
$B$ is the submodule generated by $(\varphi(f_i), -f_i)$ for all $i =
1, \dots, d$. Then $N$ is generated at most by $n-1$ elements and
$\varphi$ induces the following exact sequence
\[ 0 \longrightarrow M \longrightarrow N \longrightarrow R/(f_1,
\dots, f_{\ell}) \longrightarrow 0, \]
which shows $[M] + [R/(f_1, \dots, f_d)] = [N]$ in $\K(C)$. 
Apply the above procedure on $N$ repeatedly until it reduces to a
cyclic module. We then obtain (\ref{Srinivas}) with some $\pp m
$-primary ideal $I$ which is the annihilator of the final cyclic
module.

Let $\mathbb G_{\bullet}^i$, $\mathbb G_{\bullet}$ and $\mathbb
M_{\bullet}$ be the resolutions of $R/I_i$, $R/I$ and $M$
respectively. By assumption, the second coefficients of the
Hilbert-Kunz function of $R$ with respect to $I_i$ and $I$ all
vanish. From the above computation of the Hilbert-Kunz functions, we
obtain $ \ch_{d-1}( \mathbb G_{\bullet}^i ) \cap c_{d-1} =0 $ and $
\ch_{d-1}( \mathbb G_{\bullet} ) \cap c_{d-1} =0$.  This implies $
\ch_{d-1}( \mathbb M_{\bullet} ) \cap c_{d-1} =0$ by (\ref{Srinivas}). The theory of
Roberts and Srinivas~\cite{RoS} states that $\K(C(R)) = \K(C)$. Since
$M$ is arbitrary in the category $C$, we have that
\[\ch_{d-1}( \mathbb F_{\bullet} ) \cap c_{d-1} =0 \]
for all $\mathbb F_{\bullet}$ in the category $C(R)$. By definition, this
means $c_{d-1} =0$ in $\overline{\A_{d-1}(R)_{\mathbb Q} }$ and completes the
proof of the implication (a) $\Rightarrow$ (b).

Next in addition to the existing conditions, assuming that $R_{\pp p}$
is Gorenstein for any  $\pp p$ in $\Min R$, the set
of all minimal prime ideals of $R$, we prove that (c) is equivalent to
(a) and (b). 

Let $\Min R = \{ \pp q_1, \dots, \pp q_{\ell} \}$ which is also the
set of all associated prime ideals of $R$ since $R$ is
Cohen-Macaulay. Then the total quotient ring $Q(R)$ is a
zero-dimensional semi-local ring. Thus $Q(R)$ is complete and is a
direct product of complete local rings. Precisely we have 
$Q(R) \simeq R_{\pp q_1} \times \cdots \times R_{\pp q_{\ell}}$. Then

\[ \begin{array}{ccl}
\omega_R \otimes_R Q(R)  
& \simeq & (\omega_R)_{\pp q_1}  \times \cdots \times (\omega_R)_{\pp
  q_{\ell}} \\  
& \simeq & \omega_{R_{\pp q_1}}  \times \cdots \times \omega_{R_{\pp
    q_{\ell}}} \\
& \simeq &  R_{\pp q_1} \times \cdots \times R_{\pp q_{\ell}} \\
& \simeq & Q(R)
\end{array}
\]
The second last isomorphism is due to the assumption that $R_{\pp
  q_i}$ is Gorenstein. 

{\em Claim: There exists an embedding from $\omega_R$ into
  $R$.} 

\noindent {\em Proof of Claim.} It is well-known that $\omega_R$ is a 
maximal Cohen-Macaulay module. Therefore, 
there is an embedding from $\omega_R$ to $\omega_R
\otimes Q(R)$ hence an embedding to $Q(R)$ since 
$\omega_R \otimes Q(R) \simeq Q(R)$. On the other hand, $R$ is a
subring of $Q(R)$ and $\omega_R$ is a finitely generated
$R$-module. Thus by clearing finitely many denominators from a
generating set of $\omega_R$ in $Q(R)$, we have an embedding from
$\omega_R$ to $R$ by multiplying with some
nonzerodivisor. This completes the proof of the Claim.

The above claim leads to a short exact sequence $0\rightarrow \omega_R
\rightarrow R \rightarrow R/\omega_R \rightarrow 0$ and thus the
equality $[R/ \omega_R]_{d-1} = [R]_{d-1} - [\omega_R]_{d-1}$ in
$\A_{d-1}(R)$. On the other hand  recall that $\td([R]) = c_d +
c_{d-1} + \cdots + c_1 + c_0 \in \A_*(R)_{\mathbb Q}$, then
 $\td( [\omega_R] )= c_d - c_{d-1} + \cdots + (-1)^d c_0$ 
 since $R$ is Cohen-Macaulay. This implies that
$\td([R/\omega_R])= 2 c_{d-1} + 2 c_{d-3} + \cdots$ since Todd
class is additive. The dimension of $R/\omega_R$ is at most $d-1$
because for any minimal prime ideal $\pp q_i$ of $R$,
$(R/\omega_R)_{\pp q_i} =0$ since $R_{\pp q_i}$ is
Gorenstein. Therefore by the top term property of the Todd class which
states, in the current terminology, that $\td _{d-1} ([R/\omega_R]) =
[R/\omega_R]_{d-1}$, we have shown that $[R/\omega_R]_{d-1} =
2c_{d-1}$.

Hence in $\overline{\A_{d-1}(R)_{\mathbb Q}}$, $c_{d-1} =
\td_{d-1}([R]) =0$ if and only if $[R/\omega_R]_{d-1} =0$ which is
equivalent to $[R]_{d-1} = [\omega_R]_{d-1}$. This proves the
equivalence of (b) and (c). And the proof of the theorem is completed.

\end{proof}

If $R$ is a domain, then $[R]_{d-1}=0$ in $\A_{d-1}(R)$ by
definition. Therefore the condition (c) in Theorem~\ref{CMcase} is
equivalent to $[\omega_R]_{d-1} =0$ in
$\overline{\A_{d-1}(R)_{\mathbb Q} }$.
But if $R$ is not a domain, then the condition (c) does not imply
$[\omega_R]_{d-1} =0 $ as shown in the following example in which we
construct a Gorenstein ring $R$ of dimension 4.
The ring $R$ satisfies $[\omega_R]_3 = [R]_3$ but it does not define a zero
class in $\overline{\A_3(R)_{\mathbb Q}}$.

We first recall the definition of the idealization of a module. Let
$(S, \pp n)$ be a local ring with the maximal ideal $\pp n$ and $N$ an
$S$-module. The {\em idealization} of $N$ is a ring, denoted $S
\ltimes N$, as an $S$-module, is equal to $S \oplus N$. 
For any two elements $(a, n)$ and $(b,m)$ in $S \oplus N$,
we define the multiplication $(a,n) (b,m) = (ab, am+bn)$. It is
straightforward to check that this multiplication gives a ring
structure on $S \oplus N$ in which $0 \oplus N$ becomes an
ideal. Moreover $S \ltimes N$ is a local ring with maximal ideal $\pp n
\oplus N$ and the ideal $0 \oplus N$ is nilpotent since by
definition $(0,m) (0,n) =(0,0)$ for any $m, n \in N$. So $\dim S
\ltimes N = \dim S$. 

\begin{example}
  Let $S = k[\{x_i: i = 1, \dots, 6 \}]/I$ where $I$ is the ideal
  generated by the maximal minors of the matrix $\begin{pmatrix} x_1 &
    x_2 & x_3 \\ x_4 & x_5 & x_6 \end{pmatrix}$. It is known that $S$
  is Cohen-Macaulay of dimension $4$ with a canonical module $\omega_S
  $ which is isomorphic to the ideal generated by $x_1$ and $x_4$ over
  $S$. Viewing $\omega_S$ as an $S$-module, we let $R$ be the
  idealization $S \ltimes \omega_S$ as defined above.  Since $R
  = S \oplus \omega_S$ is a maximal Cohen-Macaulay $S$-module, $\dim R
  =\dim S =\depth S = \depth R$. So $R$ is Cohen-Macaulay.  To verify
  that $R$ is Gorenstein, we see that $R$ is finite over $S$ and the
  canonical module $\omega_R \simeq \Hom_S(R, \omega_S) \simeq
  \Hom_S(S \oplus \omega_S , \omega_S) \simeq \Hom_S(S, \omega_S)
  \oplus \Hom_S(\omega_S, \omega_S) \simeq \omega_S \oplus S$. Such
  viewpoint of $\omega_R$ comes with a natural $R$-module structure
  and so $\omega_S \oplus S$ at the other end of the isomorphism
  sequence is also an $R$-module.  Precisely for any $a \in S$ and $y
  \in \omega_S$, the multiplications on $S \ltimes \omega_S$ by
  $(a,0)$ and $(0,y)$ induce respective multiplications on
  $\Hom_S(S\ltimes \omega_S, \omega_S)$.  One can carefully check that
  these coincide with the multiplications on $\omega_S \oplus S$ as a
  module over $R$ as an idealization of
  $\omega_S$. Hence $\omega_R \simeq R$ as $R$-modules and so
  $R$ is Gorenstein.

The ring $R$ just constructed is Gorenstein of dimension 4 but is
not a domain.
We consider the short exact sequence of $R$-modules:
\[0 \rightarrow \omega_S \rightarrow S \ltimes \omega_S \rightarrow S
\rightarrow 0 . \]
Then $[S \ltimes \omega_S]_3 = [\omega_S]_3 + [S]_3$ in $\A_3(S
\ltimes \omega_S)$. 

Furthermore since $\omega_S$ is nilpotent in $R=S \ltimes \omega_S$,
the quotient ring $R/\omega_S \simeq S$ and $R$ have exactly the same
prime ideals and rational equivalence relation. Thus $\A_*(R)$ is
isomorphic to $\A_*(S)$ and we have $[S \ltimes \omega_S]_3 =
[\omega_S]_3 + [S]_3$ holds in $\A_3(S)$. But $[S]_3 =0$ since $S$ is
a domain so $[S \ltimes \omega_S]_3 = [\omega_S]_3 = - [S/\omega_S]_3$
in $\A_3(S) = \A_3(R)$. As mentioned above $\omega_S$ is isomorphic to
the ideal generated by $x_1$ and $ x_4$.  We also know that
$\A_3(S)_{\mathbb Q}$ has dimension one and can be generated by the
class of $S/(x_1, x_4)$. This shows that $-[S/\omega_S]_3=-[S/(x_1,
x_4)]$ is nonzero in $\A_3(S)_{\mathbb Q}$.  Furthermore,
$\A_3(S)_{\mathbb Q}=\overline{ \A_3(S)_{\mathbb Q} }$ by
\cite[Example 7.9]{K16}. Hence $[R]_3 \neq 0$ in $\overline{
  \A_3(R)_{\mathbb Q} }$ as wished to show in this example.

\end{example} 


\bigskip

\section{Additive Error of the Hilbert-Kunz Function}\label{error}

This section discusses some observations that grow out of the
proof in Section~\ref{domain}.  First we define an additive error of
the Hilbert-Kunz function on a short exact sequence.  Let $ 0
\rightarrow M_1 \rightarrow M_2 \rightarrow M_3 \rightarrow 0$ be a
short exact sequence of finitely generated $R$-modules.  The
alternating sum $\hk{M_3} - \hk{M_2} + \hk{M_1} $ is called the {\em
  additive error of the Hilbert-Kunz function}. It is known from
\cite[Theorem~1.6]{Mon1}  that 
\begin{equation}\label{monskyerror}
 \hk{M_3} - \hk{M_2} + \hk{M_1} = \mathcal O(q^{d-1}) 
\end{equation}
and hence the leading coefficient $\alpha(M_i)$ is additive
\cite[Theorem~1.8]{Mon1}. We give this error a more precise approximation
in Corollary~\ref{additive}. 

We will see in Theorem~\ref{tau} below that, in order to estimate 
$\hk{M} - {\rm rank}_RM \cdot \hk{R}$, each torsion free module $M$ is 
associated with a real number $\tau([M]_{d-1})$ and that this assignment is 
compatible with the rational equivalence and hence it induces a group 
homomorphism from $\A_{d-1}(R)$ to $\mathbb R$. Then we deduce that 
$\tau$ is additive on short  exact sequences and use this fact to estimate the 
additive error of the Hilbert-Kunz function. 

\begin{theorem}\label{tau} 
Let $R$ be an excellent local domain regular in 
codimension one of characteristic p.  Assume that the residue field of $R$ is 
perfect and $\dim R = d$.  Then  
there exists a group homomorphism 
$\tau: \A_{d-1}(R) \rightarrow \mathbb R$ such that for a torsion free 
$R$-module $M$ of rank $r$,  
$\hk M = r \hk R + \tau( [M]_{d-1} ) q^{d-1} + \mathcal O(q^{d-2})$. 
\end{theorem} 

\begin{proof}
  Take the integral closure  $\overline R$ of $R$ in its quotient
  field. Let $\pp m_1, \dots, \pp m_s$ be the maximal ideals of
  $\overline{R}$. Let $g_i = [\overline R/\pp m_i: R/\pp m]$ be the degree of
  the field extensions. The completion of $\overline R$ is again normal
  and it is a product of the completions of $\overline{R}_{\pp m_i}$;
  denote $\widehat {\overline R} = R_1 \times \cdots \times R_s$ with
  $R_i = \widehat {\overline{R}_{\pp m_i} }$.  Then the following
  equalities of Hilbert-Kunz functions hold:
\begin{equation}\label{hkcomp}
 \begin{array}{ccl}
\hkr{R,I}{M} & = & \hkr{R,I}{M \otimes_R \overline{R} } +\mathcal
O(q^{d-2} )\\
 & = & \sum_{i=1}^s g_i \hkr{\overline{R}_{\pp m_i}, I \overline{R}_{\pp m_i}} {M
  \otimes_R\overline R_{\pp m_i} } +\mathcal O(q^{d-2}) \\
 &  = & \sum_{i=1}^s g_i \hkr{R_i, IR_i}{M  \otimes_R R_i}  +\mathcal O(q^{d-2}).
\end{array} 
\end{equation}
The first equality holds because of the exact
  sequence $ 0 \rightarrow K \rightarrow M \rightarrow M\otimes
  \overline R \rightarrow C \rightarrow 0$ obtained by tensoring $M$
  to the inclusion $R \hookrightarrow \overline R$ for some $K$ and
  $C$. 
Since $R$ and $\overline R$ are isomorphic to each other
  when localized at prime ideals of dimension $\geq d-1$, we know that $K$
  and $C$ have dimension at most $d-2$.  Then we apply
  Lemma~\ref{codim2} to establish the equality.  

Notice that the residue field $\kappa(R_i)$ of $R_i$ is a finite field
extension of that of $R/\pp m$, so $g_i < \infty$. 
Because $R_i$ is complete with perfect residue field,
$R_i$ is $F$-finite. 

Let $M$ be a torsion free module of rank $r$.  Then $M\otimes_R R_i$
has rank $r$ and its torsion submodule is of dimension at most $d-2$.
Let $T(M\otimes_R R_i) $ denote the torsion submodule of $M\otimes_R
R_i$. Hence $\tau( [M\otimes R_i/T (M\otimes R_i)]_{d-1} )$ exists in
the sense of \cite[Corollary~1.10]{HMM} since $R_i$ is $F$-finite for
each $i$. Continuing the computation in (\ref{hkcomp}), we see that
for each $n$
\begin{equation}\label{hkcomp2}
\begin{array}{cl}
 & \hkr{R,I} M - r\hkr{R,I} R \\  = & \sum_{i=1}^s g_i (\hkr{R_i,
  IR_i}{M\otimes R_i} - r \hkr{R_i,I}{R_i} ) +\mathcal O(q^{d-2}) \\
  = & \sum_{i=1}^s g_i \tau( [M\otimes R_i/ T(M\otimes R_i)]_{d-1} ) q^{d-1} +\mathcal O(q^{d-2} ).
\end{array} 
\end{equation}

Next we consider the following composition of maps on
the $(d-1)$-components of the Chow groups 
\[ \A_{d-1}(R) \stackrel{\delta^{-1}}{\longrightarrow} \A_{d-1}(\overline
R) \stackrel{\gamma}{\longrightarrow} \A_{d-1}(\hat{ \overline R}) =  
\A_{d-1}(R_1) \times \cdots \times \A_{d-1}(R_s)  
\longrightarrow \mathbb R . 
\]
The first map $\delta^{-1}$ is the inverse of the
isomorphism $\delta: \A_{d-1}(\overline R) \rightarrow \A_{d-1}(R)$
induced by $R \hookrightarrow \overline{R}$ since  $R$
satisfies (R1). The second map $\gamma$ is as defined in
\cite[Definition~2.2]{KaK1} for the completion is a flat map.
Each $R_i$ is a complete normal local ring so there exists a map
$\tau_i: \A_{d-1}(R_i) \rightarrow \mathbb R$ by
\cite[Corollary~1.10]{HMM}. Now the desired map $\tau$ from $\A_{d-1}(R)$ to
$\mathbb R$ can be obtained by taking appropriate composition of
maps: $(\sum_{i=1}^s g_i \tau_i) \circ \gamma \circ \delta^{-1}$.
Here note that
$ \gamma \circ \delta^{-1} ( [M]_{d-1} ) = 
( [ {M \otimes R_1}/{\tau( M\otimes R_1) } ]_{d-1}, \dots, 
   [ {M \otimes R_s}/{\tau( M\otimes R_s) } ]_{d-1}  ) . 
$
According to the computation in (\ref{hkcomp2}), for any torsion
free module $M$ of finite rank, $\tau([M]_{d-1}) = \sum_i g_i \tau_i( [M
\otimes R_i / T(M\otimes R_i) ]_{d-1} )$.
\end{proof}

\begin{remark}\label{add_remark} 
  Assume that $R$ is as in Theorem~\ref{tau}. In the rest of this
  paper, for an $R$-module $M$, we denote $\tau(
  [M]_{d-1} )$ simply by $\tau(M)$. \\
  $(a)$ The map $\tau$ is additive on a short exact sequence. This is
  an immediate consequence
  of Theorem~\ref{thmChan}(b) and Theorem~\ref{tau}. \\
  $(b)$ For an arbitrary module $M$ of rank $r$, let $T$ be the
  torsion submodule of $M$ and $M'$ be the quotient $M/T$. Then
\begin{equation}\label{HKformula}
\hk M = r\hk R + (r\beta(R) + \tau(  M' ) ) q^{d-1} + \hk T+ \mathcal (q^{d-2}).
\end{equation}
This can be done by taking normalization as done in the proof of Theorem~\ref{tau}
and applying \cite[Lemma~1.5]{HMM}. (Or see a straightforward argument in 
Theorem~\ref{thm:HKfunc} of the next section.)

\end{remark}

\begin{corollary}\label{additive} Assume that $R$ is as in
  Theorem~\ref{tau}. Let $ 0 \rightarrow M_1 \rightarrow M_2
  \rightarrow M_3 \rightarrow 0$ be a short exact sequence of finitely
  generated $R$-modules. For each $i$ let $T_i$ be the torsion
  submodule of $M_i$ and $M_i'$ the torsion free module $M_i/T_i$.  Then
  \[ \hk{M_3} - \hk{M_2} + \hk{M_1} = \hk{T_3/ \delta_2(T_2)} -
  \tau( T_3/\delta_2(T_2) ) q^{d-1} + \mathcal O (q^{d-2}) \] where 
    $\delta_2$ is the induced map from $T_2$ to $T_3$.
\end{corollary}

\begin{proof} 
For each $i$, letting $r_i$ be the rank of $M_i$, we have 
\begin{equation*}
\hk{M_i} = r_i \alpha(R) q^d + (r_i \beta(R) + \tau(M_i') ) q^{d-1}  + \hk{T_i} + \mathcal O(q^{d-2}) 
\end{equation*} 
by (\ref{HKformula}). Therefore
\[ \begin{array}{cl}
    &  \hk{M_3} -\hk{M_2} + \hk{M_1} \\ 
= & 
( \hk{T_3} - \hk{T_2} + \hk{T_1} ) + 
(\tau(M_3' ) -  \tau(M_2') +
\tau(M_1') ) q^{d-1} + \mathcal O (q^{d-2}) .
\end{array}
\]

Next we inspect  $\hk{T_i}$ and 
the alternating sum of $\tau(M_i)$ in order to estimate the additive
error of the Hilbert-Kunz function on a short exact sequence. 

The submodules $T_i$ and $M_i'$ give the following exact
sequences:
\[ \begin{array}{rcl}
  0 \longrightarrow T_1 \stackrel{\delta_1}{\longrightarrow}  & T_2 & 
  \stackrel{\delta_2}{\longrightarrow} T_3  \\
  0 \longrightarrow M_1'  \stackrel{f_1}{\longrightarrow} & M_2' &  \\
  & M_2' & \stackrel{f_2}{\rightarrow} M_3' \rightarrow 0 .
\end{array}
\]
The kernel of the map from $M_2'$ to $M_3'$ contains $f_1(M_1')$ but
is not necessarily equal. Therefore $T_i$ and $M_i'$ do not form their
own short exact sequences. By the standard diagram chasing, one sees
$T_3/ \delta_2(T_2) \cong \operatorname{Ker} f_2 / \operatorname{Im}
f_1$.

For the simplicity of notation, we identify $T_1$ with its image in $T_2$ via $\delta_1$. 
The above exact sequence of $T_i$ induces an injection from $T_2/T_1$
to $T_3$. Note that
$ 0 \longrightarrow T_2/ T_1 \longrightarrow T_3 \longrightarrow T_3/\delta_2(T_2)
\longrightarrow 0 $
is exact. By tensoring with $R/I_n$, we have the following exact sequences
\[ \cdots \longrightarrow \tor_1(T_2/ T_1, R/I_n) \stackrel{\eta_1}{
  \longrightarrow} T_1 \otimes R/I_n \longrightarrow T_2 \otimes R/I_n
\longrightarrow T_2/ T_1 \otimes R/I_n \longrightarrow 0 \]

\[ \cdots \longrightarrow \tor_1( T_3/\delta_2(T_2), R/I_n)
\stackrel{\eta_2}{\longrightarrow} T_2/T_1 \otimes R/I_n
\longrightarrow T_3 \otimes R/I_n \longrightarrow T_3/\delta_2(T_2) \otimes R/I_n
\longrightarrow 0 .\]
This implies
\[ \hk{T_2/T_1} - \hk{T_2} + \hk{T_1} - \ell(H_1) =0 \]
and
\[  \hk{ T_3/\delta_2(T_2) } - \hk{T_3} + \hk{ T_2/T_1} - \ell(H_2) =0 \]
where $H_1$ and $H_2$ are the images of $\eta_1$ and $\eta_2$
respectively. Thus
\[ \begin{array}{rcl} 
 \hk{T_3} - \hk{T_2} + \hk{T_1} 
& = & \hk{ T_3/\delta_2(T_2) } + \hk{T_2/T_1} - \ell(H_2) - \hk{T_2/T_1} + \ell(H_1)  \\
& = & \hk{ T_3/\delta_2(T_2) } + \ell(H_1) -\ell(H_2). 
\end{array} 
\]

Let $S$ be the quotient of $R$
modulo the annihilator of $T_2$. Then $\dim S= \dim T_2 \leq d-1$. 
It is straightforward to see that $\hkr{S}{T_2} = \hkr{R}{T_2}$. This
implies, by (\ref{monskyerror}),
\[\begin{array}{ccl}
 \ell(H_1) & = & \hk{T_2/T_1} - \hk{T_2} + \hk{T_1} \\
       & = & \hkr{S}{T_2/T_1} - \hkr{S}{T_2} + \hkr{S}{T_1}  \\
       & = & \mathcal O(q^{\dim T_2 -1}). 
\end{array}
\]
Similarly we have $\ell(H_2) = \mathcal O(q^{\dim T_2 -1})$. Hence 
\[ \hk{T_3} - \hk{T_2} + \hk{T_1} = \hk{T_3/\delta_2(T_2)} + \mathcal O
(q^{ \dim T_2 -1})  . \]

For the torsion free part, let $K$ be the module such that 
$ 0 \longrightarrow K \longrightarrow M_2' \stackrel{f_2}{\longrightarrow} M_3'
\longrightarrow 0 $
is exact. By Theorem~\ref{thmChan}(b), $[M_2']_{d-1} - [M_3']_{d-1} =
[K]_{d-1}$.  Then 
\[ [M_3']_{d-1} - [M_2']_{d-1} + [M_1']_{d-1} = - [K]_{d-1} + [
M_1']_{d-1} =- [T_3/\delta_2(T_2)]_{d-1} .\] 
The second last equality holds because
$T_3/\delta_2(T_2) \cong \operatorname{Ker} f_2/ \operatorname{Im} f_1$, and
$\operatorname{Ker} f_2 = K$ and $ \operatorname{Im} f_1 \cong
M_1'$. Theorem~\ref{tau} shows that $\tau$ is a group homomorphism so
we obtain
\[ \begin{array}{rcl}
\tau(M_3') - \tau(M_2') +\tau( M_1') & = & 
\tau([M_3']_{d-1}) - \tau([M_2']_{d-1}) +\tau( [ M_1'] _{d-1} ) \\
& = & \tau( [M_1']_{d-1}  - [K]_{d-1} )  \\ 
& = & - \tau( [T_3/\delta_2(T_2)]_{d-1})  =- \tau( T_3/\delta_2(T_2)  )
\end{array} \]

It is clear now that 
\[ \begin{array}{cl}
    &  \hk{M_3} - \hk{M_2} + \hk{M_1} \\ 
= &
( \hk{T_3} - \hk{T_2} + \hk{T_1} ) + 
(\tau(M_3 ) -  \tau(M_2') +
\tau(M_1') ) q^{d-1} + \mathcal O (q^{d-2}) \\
= & \hk {T_3/\delta_2(T_2)} - \tau( T_3/\delta_2(T_2)  ) q^{d-1} + \mathcal O(q^{d-2}).
\end{array}
\] 
\end{proof}

As an immediate corollary, if $R$ is an integral domain satisfying the
assumption in Corollary~\ref{additive} and $\A_{d-1}(R)=0$, then the additive error of the
Hilbert-Kunz function is measured by 
$\hk{T_3/\delta_2(T_2) }$ up to $\mathcal O(q^{d-2})$.

\begin{example} 
  Let $R = k[x_1, \dots, x_6] / I_2$ where $I_2$ indicates the ideals
  generated by the $2 \times 2$ minors of the matrix $ \begin{pmatrix}
    x_1 & x_2 &x_3 \\ x_4 &x_5 & x_6 \end{pmatrix}$. The Hilbert-Kunz
  function of $R$ is $\hk R= \frac{1}{8} (13 q^4 -2 q^3 - q^2 -2q)$
  computed by K.-i. Watanabe. We consider the exact sequence
  \[ 0 \rightarrow (x_1, x_2, x_3) \rightarrow R \rightarrow R/(x_1,
  x_2, x_3) \rightarrow 0.\] It is known that $\omega_R=(x_1, x_4)$ is
  its canonical module and $[R/(x_1, x_4)]$ and $-[R/(x_1,x_2,x_3)]$
  are rationally equivalent in $A_3(R)$. Thus
  $\tau([R/(x_1,x_2,x_3)])=\tau(\omega_R) = \frac{1}{2}$ since $
  -\frac{1}{2} \tau(\omega_R)=\beta (R) $ by Kurano~\cite{K5}. By
Corollary~\ref{additive}, the additive error of the above short exact
sequence is 
\[\begin{array}{ccl}
&  & \hk{ T_3/ \delta_2(T_2) } - \tau ( T_3/ \delta_2(T_2) )  q^2+
\mathcal O ( q^2 ) \\
& = & \hk {R/(x_1,x_2,x_3) } - \tau( R/(x_1,x_2,x_3)  ) q^3+
\mathcal O(q^{2})  \\ & = &
 q^3 - \frac{1}{2} q^3 + \mathcal O(q^2) \\
&  = & \frac{1}{2} q^3 + \mathcal O(q^2).
\end{array} \]

\end{example}


\bigskip


\section{Appendix: $F$-finite Integral Domain with (R1$'$) } \label{domain}

We have seen, in Section~\ref{general}, the existence of the second
coefficient of the Hilbert-Kunz function. In the current section, we
revisit the proof of Huneke, McDermott and Monsky~\cite{HMM}, in which
the second coefficient is analyzed in details. Each lemma in
\cite{HMM} is a special case of its own interest in terms of  the
result and the proof.

The proof in \cite{HMM} utilizes the divisor class group and a key
lemma that states that if $T$ is a finitely generated torsion module
over a local ring $(R, \pp m)$ of characteristic $p$, then $\ell(
\tor_1^R(T, R/I_n)) = O(q^s)$ where $I$ is an $\pp m$-primary ideal
and $s=\dim T$ ({\em c.f.} \cite[Lemma~1.1]{HMM}). Note that this key lemma holds
for general local rings of positive characteristic. When a ring is no
longer normal, the Chow group is often a natural extension of the
divisor class group. Also the corresponding divisor class is 
additive on short exact sequences in the Chow group (see
Theorem~\ref{thmChan}) and this fact is used in proving several lemmas
in \cite{HMM}.

Hence the purpose of this section is twofold. We
present a generalized proof of \cite{HMM} for an integral domain under
the assumption that the domain is $F$-finite and regular in
codimension one. To do so, the Chow group is utilized in place of the
divisor class group. Thus our purpose is also to demonstrate how rational 
equivalence can be applied in this study. Although the proof presented here 
follows a similar structure as \cite{HMM}, the argument is a nontrivial extension. 
The limitation of this argument is that the authors do not know if it is possible to
extend each lemma for a ring that is not necessarily an integral domain. Hence
the main result Theorem~\ref{thm:HKfunc} requires stronger assumption on the 
ring than Theorem~\ref{thm.HK}.  

The idea of the proof goes as follows: we first focus on torsion free
modules and prove that their Hilbert-Kunz function is different from
that of free modules of the same rank by a function in form of $\tau
q^{d-1} + \mathcal O(q^{d-2})$. The constant $\tau$ depends on the
cycle class of the module $M$ and is zero if $M$ defines a zero class.
(Note that this constant $\tau$ is the same as $\tau(M)$ discussed in
Remark~\ref{add_remark}.) Then for an arbitrary module $M$, we put
together the information for the functions of the torsion submodule
$T$ and the torsion free module $M/T$ to obtain $\hk M$.  We remind
the readers to recall the definition of rational equivalence and cycle
classes from Section~\ref{chow}.

For the remaining of the paper, we assume that $R$ is a local integral
domain with perfect residue class field such that the Frobenius map
$f: R \rightarrow R$ is finite and $R$ is regular in codimension one
as defined in the introduction.  By Kunz's theorem~\cite{Kun},
$F$-finiteness implies that $R$ is excellent hence the integral
closure $\overline R$ is finite over $R$.

\begin{lemma}\label{zeroclass}
  Let $R$ be an $F$-finite local domain regular in codimension one
  with perfect residue class field.  Let $J$ be a nonzero ideal in
  $R$. Assume $J$ defines the zero class in $\A_{d-1}(R)$; {\em i.e.},
  $[J]_{d-1} =0$. Then $\hk J = \hk R + \mathcal O(q^{d-2})$.
\end{lemma}

\begin{proof} 
  By assumption $[J]_{d-1} =0$, there exist nonzero elements $a$ and
  $b$ in $R$ such that $[J]_{d-1} = \divs(\pp o; b) - \divs(\pp o; a)$
  where $\pp o$ denotes the zero ideal. By definition of divisors, this
  implies
\begin{equation}\label{eq:length}
 \ell(R_{\pp p}/ J R_{\pp p}) + \ell (R_{\pp p}/ a R_{\pp p}) =
\ell( R_{\pp p}/b R_{\pp p} ) 
\end{equation}
for any prime ideal $\pp p$ of dimension $d-1$.  Since $R$ is regular
in codimension one, $R_{\pp p}$ is a DVR and every ideal is a power of
the maximal ideal. We have that $\ell(R_{\pp p}/ J R_{\pp p}) + \ell
(R_{\pp p}/ a R_{\pp p}) = \ell(R_{\pp p}/ aJ R_{\pp p})$ so
(\ref{eq:length}) is equivalent to $aJR_{\pp p} = bR_{\pp p}$.

Thus as ideals in $R$, $aJ$ and $bR$ have the same primary decomposition up to codimension 
one. Namely, there exist prime ideals $\pp q_1, \dots, \pp q_s$ of height one and ideals $Q$, $Q'$ 
of height two such that
\begin{eqnarray*}
 aJ & = & \pp q_1 \cap \cdots \cap \pp q_s \cap Q \\
 bR & = &\pp q_1 \cap \cdots \cap \pp q_s \cap Q' .
\end{eqnarray*}
Note that $\dim (\pp q_1 \cap \cdots \cap \pp q_s / aJ ) \leq d-2$ and similarly for the ideal $bR$. 
Moreover since $a$ and $b$ are both nonzerodivisor, $aJ \cong J$ and $bR \cong R$. 
By Lemma~\ref{codim2}
\begin{eqnarray*} 
\hk J = \hk{aJ} & = & \hk{\pp q_1 \cap \cdots \cap \pp q_s } + \mathcal O(q^{d-2})  \\
 &  =  & \hk{bR } +\mathcal O(q^{d-2}) \\
  & =  & \hk{R} + \mathcal O(q^{d-2}). 
\end{eqnarray*}
\end{proof}

The following lemma shows that if a torsion
free module has zero cycle class, then its Hilbert-Kunz function is
compatible with that of the free module of the same rank.

\begin{lemma}\label{torfreezeromod}
 Let $R$ be as in Lemma~\ref{zeroclass}. Let $M$ be a finitely generated torsion free
  module of rank $r$.  Assume that $[M]_{d-1} =0$ in
  $\A_{d-1}(R)$. Then $\hk M = r \hk R + \mathcal O(q^{d-2}) $.
\end{lemma}

\begin{proof} 
 That $M$ is torsion free of rank $r$ implies a sequence of
  inclusions for a fixed basis: $M \subset R^r \subset {\overline
    R}^r$. Let $m_1, \dots, m_s$ be a set of generators for $M$ and
  let $\overline M$ denote the submodule $\overline Rm_1 + \cdots +
  \overline R m_s$ of $ {\overline R} ^r$.  We observe that $\dim \overline M/M \leq
  d-2$ since $\overline{M}_{\pp p} = M_{\pp p}$ for any prime of
  dimension $d$ or $d-1$. In the normal ring $\overline R$, there exists an ideal $\pp a$
  that results the following short exact sequence

\begin{equation}\label{ses}
 0 \longrightarrow \overline R^{r-1} \longrightarrow \overline M \longrightarrow \pp a \longrightarrow 0. 
 \end{equation}
Indeed $\pp a$ is a Bourbaki ideal of $\overline M$ with respect to a free submodule of rank $r-1$. 
All modules in (\ref{ses}) are finitely generated $R$-modules and the sequence remains exact
as the homomorphisms are viewed over $R$. Tensoring (\ref{ses})  
by $R/I_n$ over $R$, one obtains an exact sequence
\[ 0 \longrightarrow L \longrightarrow (\overline R/I_n \overline R)^{r-1} 
\stackrel{\phi}{\longrightarrow} \overline M/ I_n \overline M
 \longrightarrow \pp a/I_n \pp a \longrightarrow 0 \]
where $L$ indicates the kernel of the homomorphism $\phi$. This shows 
\begin{equation}\label{ineq1}
 \hk{\pp a} - \hk{\overline M} + (r - 1) \hk{\overline R} = \ell( L)\geq 0 .
 \end{equation}

Recall that $\dim \overline R/R \leq d-2$ and $\dim \overline M/M \leq
d-2$. Thus applying Theorem~\ref{thmChan}(2), we have 
$[\overline R]_{d-1} =[R]_{d-1} =0$ and $[\overline M]_{d-1}
=[M]_{d-1} =0$ in $\A_{d-1}(R)$ which further implies $[\pp a]_{d-1} =
[\overline M]_{d-1} -(r-1) [\overline R]_{d-1} =0$. 
Initially $\pp a$ is an ideal of the integral domain
$\overline R$ so as an $R$-module, $\pp a$ is also finitely
generated torsion free of rank one. Also there exists an element $c$ in $R$
such that $c\pp a \subset R$ as an ideal and $\hk{c \pp a} =
\hk{\pp a}$ since $\pp a$ is isomorphic to $c\pp a$.  Moreover 
$[c \pp a]_{d-1} =[\pp a]_{d-1}  =0$. 
By (\ref{ineq1}), Lemmas~\ref{zeroclass} and \ref{codim2}, we have
\begin{equation*}
 \hk{\overline M}   \leq \hk{\pp a} + (r-1) \hk{\overline R} = r
\hk{R} + \mathcal O(q^{d-2})  
\end{equation*}
and
\begin{equation}\label{ineq2}
  \hk M   \leq r \hk R + \mathcal O(q^{d-2}) . 
\end{equation}

On the other hand there exists a short exact sequence $0 \rightarrow L
\rightarrow R^{r+s} \rightarrow M \rightarrow 0$ by resolution. The
module $L$ is torsion free of rank $s$ and $[L]_{d-1} =0$ again by
Theorem~\ref{thmChan}(b). We obtain another exact sequence by
tensoring with $R/I_n$:
\[ 0 \longrightarrow L' \longrightarrow L/I_n L \longrightarrow
(R/I_n)^{r+s} \longrightarrow M/I_nM \longrightarrow 0.  \] By
(\ref{ineq2}), $\hk L \leq s \hk R + \mathcal O (q^{d-2})$. A
similar computation as above shows
\[ \hk M \geq (r+s) \hk R - \hk L \geq  r  \hk R +\mathcal O(q^{d-2}) .\]
Hence $\hk M = r \hk R + \mathcal O(q^{d-2})$ and the proof is completed. 
\end{proof}

\begin{lemma}\label{torsionfree} 
Let $R$ be as in Lemma~\ref{zeroclass}.
Let $M$ be a finitely generated torsion free $R$-module. 

$(a)$ If $N$ is a finitely generated torsion free $R$-module such that
$[N] = [M] $ in $\A_d(R) \oplus \A_{d-1}(R)$, then $ \hk M = \hk N +
\mathcal O(q^{d-2}) $.

$(b)$ $\ell( \tor_1 ( M, R/I_n) ) = \mathcal O(q^{d-2})$. 
\end{lemma} 

\begin{proof} 
  The assumption $[M] = [N]$ indicates that $M$ and $N$ have the same
  rank and $[M]_{d-1} = [N]_{d-1}$ in $A_{d-1}(R)$.  We write
  $[M]_{d-1} = \sum_{i=1}^s [R/ \pp p_i] $. Then the module $M
  \oplus (\oplus_i \pp p_i)$ is a torsion free module of rank $r+s$
  and determines a zero class in $\A_{d-1}(R)$ since $[\pp p_i]_{d-1}
  = - [R/\pp p_i]$. By Lemma~\ref{torfreezeromod} $M \oplus (\oplus_i
  \pp p_i)$ has the Hilbert-Kunz function in the form of $(r+s) \hk{R}
  + \mathcal O(q^{d-2})$ and similarly for $N$. Thus we have
\[ \hk M = (r+s) \hk R - \sum_{i=1}^s \hk{\pp p_i} +   
   \mathcal O(q^{d-2})  = \hk{N} \mathcal + O(q^{d-2}) .\]

To prove (b), we consider a short exact sequence $0 \rightarrow
G \rightarrow F \rightarrow M \rightarrow 0$ by resolution where $F$
is a free module. This induces the exact sequence
\[ 0 \longrightarrow \tor_1(M, R/I_n) \longrightarrow G/I_n G
\longrightarrow F/I_n F \longrightarrow M/I_nM \longrightarrow 0 . \]
Since $[F]_{d-1} = [ G \oplus M]_{d-1}$ and both modules are torsion
free, by the result of Part~(a), we conclude that
\[ \ell(\tor_1(M, R/I_n) ) = \hk G - \hk F + \hk M  = \mathcal O(q^{d-2}). \]

\end{proof}

Lemma~\ref{torsionfree} plays an important role in the discussion below
that leads to Theorem~\ref{thm:HKfunc}.  

To better analyze the Hilbert-Kunz functions for all finitely generated modules, the
following definition is given to a torsion free module $M$ of rank $r$: 
\[ \delta_n(M) = \hk M - r \hk R. \] The function $\delta_n(M)$ mainly
measures the difference between the Hilbert-Kunz function of $M$ and
that of a free module of the same rank. Notice that $\delta_n(M)$ is a
function of $\mathcal O(q^{d-1})$ and it is of $\mathcal O(q^{d-2})$
if $[M]_{d-1} =0$ by Lemma~\ref{torfreezeromod}. In particular
$\delta_n(R) =0$ and $\delta_n(M \oplus N) = \delta_n(M) + \delta_n(N)
$. If $[M]= [N]$ in $\A_d(R) \oplus \A_{d-1}(R)$, then $\delta_n(M) =
\delta_n(N) + \mathcal O(q^{d-2})$ by Lemma~\ref{torsionfree}~($a$).
In fact as already proved in (\ref{hkcomp2}), $\delta_n(M) = \tau
q^{d-1} + \mathcal O(q^{d-2})$ for some constant $\tau$ by taking
normalization.  Here we show that this can be achieved independently
within the current framework using the next Theorem~\ref{deltafcn}
which gives a recursive relation on $\delta_n(M)$ for a given $M$ via
the Frobenius map.

We recall the Frobenius map $f: R \rightarrow R$ on a ring $R$ of
characteristic $p$ assuming $R$ is complete with perfect residue class
field, for any $x$ in $R$, $f(x) = x^p$. For any $R$-module $M$, $^1M$
denotes the same additive group $M$ with an $R$-module structure via
$f$. Since $f$ is a finite map, $^1R$ is a torsion free $R$-module of
rank $p^d$.  If $M$ is a module of rank $r$ then $^1M$ has rank $p^d
r$ over $R$.  The map on the Chow group $f^*: \A_{d-1}(R) \rightarrow
\A_{d-1}(R)$ induced by the Frobenius map is multiplication by
$p^{d-1}$.  We observe $I_n \cdot _f M = (I^{[p^n]})^{[p]}M =
I_{n+1}M$. Therefore $\hk{^1M} = \varphi_{n+1}(M)$.

\begin{theorem}\label{deltafcn} 
  Let $R$ be an $F$-finite local domain regular in codimension one
  with perfect residue class field.  Let $M$ be a torsion free $R$-module
  of rank $r$. Then $\delta_{n+1}(M) = p^{d-1} \delta_n(M) + \mathcal
  O(q^{d-2})$.
\end{theorem} 

\begin{proof} 
First we claim that $[^1M]_{d-1} = p^{d-1} [M]_{d-1} +
r[^1R]_{d-1}$. Indeed there exists an embedding of a free module $F$
of rank $r$ into $M$ such that $M/F$ is a torsion module over $R$.
The sequence $0 \rightarrow {^1F}
\rightarrow {^1M} \rightarrow {^1(M/ F)} \rightarrow 0$ remains 
exact as $R$-modules  via $f$ by restriction.
We have  $[^1M]_{d-1} = [^1(M/F)]_{d-1} + r[^1R]_{d-1} $ in
$\A_{d-1}(R)$.  Notice that $^1F$ also has rank $p^dr$ so $^1(M/F)$ is
torsion, and $[^1(M/F)]_{d-1} = f^*([M/F]_{d-1})$  
since $\dim \, ^1(M/F) \leq d-1$.  
Furthermore $[M/F]_{d-1} = [M]_{d-1}$ by the above exact
sequence since $[R]_{d-1} =0$.  This shows $[^1(M/F)]_{d-1} = f^*([M]_{d-1}
)= p^{d-1}[M]_{d-1}$ and $[^1M]_{d-1} = p^{d-1} [M]_{d-1} +
r[^1R]_{d-1}$  as claimed.

Now we consider $^1M \oplus R^{p^{d-1}r}$ and $(M^{p^{d-1}}) \oplus
( {^1R})^r$. An extra copy of a free module is added to $^1M$
so that both modules have the same rank. These two modules define the same
cycle class in $\A_d(R) \oplus \A_{d-1}(R)$ by the above claim. Since
$M$ is torsion free, obviously both modules above are torsion
free and so their Hilbert-Kunz functions coincide up to $\mathcal O(
q^{d-2} ) $ by Lemma~\ref{torsionfree}($a$).  The expected result follows
straightforward computations and that $\hk{^1M} = \varphi_{n+1}(M)$:
\begin{eqnarray*}
 \hk{^1M} +  {p^{d-1}r} \hk R & = & {p^{d-1}}  \hk M + r \hk{^1R}  + \mathcal O(q^{d-2}) \\
 \varphi_{n+1}(M) - r \varphi_{n+1}(R)  & =  & {p^{d-1}}  \hk M -
 {p^{d-1}r} \hk R + \mathcal O(q^{d-2}).
\end{eqnarray*}
Hence $\delta_{n+1}(M) = p^{d-1} \delta_n(M) + \mathcal O(q^{d-2})$.
\end{proof}

\begin{lemma}\label{deltatau} 
Let $R$ be as in Theorem~\ref{deltafcn}. Assume $R$ has perfect residue
field and $\dim R = d$. 
Let $M$ be a torsion free module of rank $r$. Then there is a real constant
$\tau(M)$ such that $\delta_{n}(M) = \tau(M) q^{d-1}  + \mathcal O(q^{d-2})$.
\end{lemma} 

The proof of Lemma~\ref{deltatau} using Theorem~\ref{deltafcn} is very
similar to the original one \cite[Theorem~1.9]{HMM}. A quick sketch can
be done by setting $v_n(M) = \displaystyle{\frac{\delta_n(M) }{
    q^{d-1} }} $. Notice that $q=p^n$ varies as $n$ does. By careful
yet straightforward computations, one shows that $v_{n+1}(M) - v_n(M)
= \mathcal O(\frac{1}{p^n} ) $. Hence $v_n(M)$ converges to some
constant, denoted $\tau(M)$, as $n \longrightarrow \infty$. As indicated 
earlier in Section~\ref{error}, the value $\tau(M)$ is the same as $\tau([M]_{d-1})$ 
in Theorem~\ref{tau}. Discussions on the properties of $\tau(M)$ and 
its application can also be found in Section~\ref{error}.

We conclude this appendix by describing the first and second coefficients of
$\hk M$ for an arbitrary finitely generated module.  For an arbitrary
torsion free module $M$ of rank $r$, we have a comparison of its
Hilbert-Kunz function with that of a free module of the same rank
as an immediate corollary of Lemma~\ref{deltatau}:

\begin{equation}\label{torsionfree2eq}
\hk M = r \hk R + \tau(M) q^{d-1} + \mathcal O(q^{d-2}) .
\end{equation}

\begin{theorem}\label{thm:HKfunc} 
Let $R$ be as in Theorem~\ref{deltafcn}. 
Let $T$ be the torsion submodule of $M$. Then the coefficients 
$\alpha(M)$ and $\beta(M)$ in the Hilbert-Kunz function $\hk{M}$ of $M$ have 
the following properties:  \\
(a) $\alpha(M) = r \alpha(R)$;  \\
(b) $\beta(M) = r \beta(R) + \beta(T) + \tau(M/T)$. 
\end{theorem} 
\begin{proof} 
We set $M' = M/T$. The short exact sequence
$0 \rightarrow T \rightarrow M \rightarrow M' \rightarrow 0$
induces the  following long exact sequence 
\[ \cdots \longrightarrow \tor_1(M', R/I_n) \longrightarrow T/I_nT 
\longrightarrow M/I_nM \longrightarrow M'/I_nM' \longrightarrow 0 . \]
In the above $M'$ is a torsion free module. By
Lemma~\ref{torsionfree}($b$), we have $\hk M  = \hk T + \hk{M'} +
\mathcal O(q^{d-2})$.  Note that $M'$ is a torsion free module of rank $r$ and 
$\dim T \leq d-1$ since $T$ is torsion.  Thus using
Monsky's original work~\cite{Mon1} for $T$, (\ref{torsionfree2eq}) for $M'$ and
Theorem~\ref{thm.HK} for $R$, we have
\[ \begin{array}{ccl}
\hk M & = &  \hk T   + r \hk R + \tau(M') q^{d-1} + 
 \mathcal O(q^{d-2}) \\
& = & \beta(T) q^{d-1} + r \alpha (R) q^d + r \beta(R)   q^{d-1} + \tau(M') q^{d-1} +
\mathcal O(q^{d-2})  \\
& = & r \alpha(R) q^d + (r \beta(R) + \beta(T) + \tau(M') ) q^{d-1} + \mathcal O(q^{d-2}). 
\end{array} 
\]  
Hence the constants $\alpha(M)$ and $\beta(M)$ have the desired forms. 
\end{proof}


\medskip

{\small 
\noindent Department of Mathematics, Central Michigan University, Mt. Pleasant, MI~48859, U.S.A.

\noindent {\em email}: chan1cj@cmich.edu

\noindent Department of Mathematics, School of Science and Technology,
Meiji University, Higashimita 1-1-1, Tama-ku, Kawasaki 214-8571, Japan

\noindent {\em email}: kurano@isc.meiji.ac.jp
}

\end{document}